\documentclass[leqno,a4paper]{amsart}

\usepackage{soul}

\usepackage{amsmath,amsfonts,amsthm,amssymb,dsfont}

\usepackage{times}
\usepackage{microtype}

\usepackage{cite}

\usepackage{mathrsfs}
\usepackage{enumerate, xspace}
\usepackage{graphicx}
\usepackage{color}
\usepackage[makeroom]{cancel} 

    \usepackage[all, knot]{xy}
    \xyoption{arc} 

\usepackage{pgf,tikz}\usepackage{mathrsfs}\usetikzlibrary{arrows}

\usepackage{verbatim}

\usepackage{caption}
\usepackage{subcaption}
\usepackage[pdfauthor={Arman Darbinyan and Markus Steenbock},pdftitle={Embeddings into left-ordered simple groups},pdfkeywords={embedding theorems, left-ordered groups, simple groups, word problem, computable algebra},pdftex]{hyperref}

\hypersetup{colorlinks,%
citecolor=black,%
filecolor=black,%
linkcolor=black,%
urlcolor=black,%
}

\theoremstyle{plain}

				\newtheorem{quest}{Question}
				
				\newtheorem{thm}{Theorem}[section]
				\newtheorem{prop}[thm]{Proposition}
				\newtheorem{lem}[thm]{Lemma}

\theoremstyle{definition}		\newtheorem{df}[thm]{Definition}

\newtheorem{rem}[thm]{Remark}

	\theoremstyle{remark} 	

\title{On what finitely generated (left-orderable) simple groups can know about their subgroups}
\author{Arman Darbinyan}
\address{{University of Southampton, UK}}
\email{a.darbinyan@soton.ac.uk}
\author{Markus Steenbock }
\address{Universit\"at Wien, Austria}
\email{markus.steenbock@univie.ac.at}


\begin{document}

\begin{abstract}
In this paper, we survey some of the recent advances on embeddings into finitely generated (left-orderable) simple group such that the overgroup preserves algorithmic, geometric, or algebraic information about the embedded group. We discuss some new consequences and also extend some of those embedding theorems to countable classes of finitely generated groups.
\end{abstract}

 \maketitle

 \section{Introduction}
Studies on embeddings of countable groups into finitely generated groups that are either simple or left-orderable have a long history. For example, a famous theorem of Neumann states that any countable left-orderable group embeds into a 2-generated left-orderable group \cite{neumann_embedding_1960}. A theorem of Hall, whose proof, similar to the proof of Neumann, uses wreath product techniques, states that every countable group embeds into a finitely-generated simple group \cite{hall_embedding_1974}. Later, using techniques based on so called free-constructions, alternative proofs to the theorem of Hall were obtained by Gorjushkin \cite{gorjuskin_imbedding_1974} and Schupp \cite{schupp_embeddings_1976}. A natural question on such embeddings asks what extra properties those embeddings can preserve? A famous result in this direction is by Thompson 
\cite{thompson_word_1980}, which states that every finitely generated group $G$ with decidable word problem Frattini embeds into a finitely-generated simple group $H$ that also has decidable word problem.  
 In particular, this means that $H$ does not distort the conjugacy structure of ${G}$. 
The {original} proof of Thompson's theorem uses properties of homeomorphism groups of Cantor sets {in an essential way}.

 \subsection{Embeddings into simple groups}
 
  Motivated by the aforementioned results, we ask what other properties can embeddings into finitely-generated (left-orderable) simple groups preserve? In addition, it is natural to ask which groups embed into finitely-generated groups that are simultaneously left-orderable and simple, and what properties can such embeddings preserve? The theorems formulated below are contributions in those directions. Theorem \ref{thm-0}, Proposition \ref{prop-0} and Theorem \ref{thm-finitization} directly follow from \cite{ourpaper}. Theorems \ref{thm-new-2}, \ref{T: family computable left order} and \ref{T: family word problem} are proved in  Sections 2 and 3.
 
 \begin{thm}[Darbinyan-Steenbock, \cite{ourpaper}]
 \label{thm-0}
 If $G$ is a countable left-orderable group, then $G$ has a Frattini embedding into a finitely generated left-orderable simple group $H$. Moreover, if $G$ is finitely generated, then the embedding can be chosen to be isometric.
 \end{thm}
  
An embedding $\phi$ of $G$ into $H$ is called \emph{Frattini} if any two elements of $G$ are conjugate in $G$ if and only if their images under $\phi$ are conjugate in $H$. In this case, we also say that \emph{$G$ Frattini embeds into $H$}. One can think about a {Frattini embedding} $G \hookrightarrow H$ as such that the group $H$ carries extra algebraic knowledge about the embedded group $G$ that is reflected in preserving the conjugacy structure of the embedded group by the overgroup.  Similarly, the property that the embedding is \emph{isometric}, given that $G$ is finitely generated, in layman's term would mean that the large group $H$ carries knowledge about the geometry of the group $G$.  Namely, with respect to certain finite generating sets of $H$ and $G$, the word-metric distance between elements of $H$ coincides with the word-metric distance with respect to the word metric of $G$. Alternatively, one can say that $H$ knows the word metric geometry of $G$. A more relaxed version of the isometric embedding is the  \emph{quasi-isometric} embedding. The embedding of $G$ into $H$ is quasi-isometric, if the word metric of $G$ is comparable to the metric induced by the word metric of $H$ up to multiplicative and additive correction terms, see Definition \ref{D:quasi-isometry} below.
 
 One can ask what other properties the large group $H$ ``knows" about $G$. In particular, one could ask about algorithmic information that $H$ knows about $G$. A left-order on a group $G$ is \emph{computable} if there is an algorithm to decide whether a given group element is larger than, smaller than or equal to the identity. In \cite{ourpaper},   we showed the following:

 \begin{prop}[Darbinyan-Steenbock, \cite{ourpaper}]
 \label{prop-0}
 If the group $G$ from Theorem \ref{thm-0} has {$a$} computable left-order, then $H$ can be chosen {to be} with a computable left-order.
 \end{prop}
 In the light of Proposition \ref{prop-0}, it is natural to ask if the group $H$ ``knows'' whether or not $G$ has computable left-order. {More precisely, knowing that $H$ has a computable left-order, does it imply that $G$ also must have a computable left-order?} The answer to this question is positive once we require the group $G$ to be finitely generated, as it follows from Lemma 2.2 in \cite{arman_new}. Namely, if $G$ is finitely generated then the existence of a computable left-order for $H$ implies the existence of such orders for $G$. 
 
For two enumerated sets $\mathcal{A}$ and $\mathcal B$,  a map $\phi: \mathcal A \rightarrow \mathcal B$ is called \emph{computable} if there exists a Turing machine that for each input $i \in \mathbb{N}$ computes the number of the image of the $i$-th element of $\mathcal A$ under $\phi$. This definition naturally extends to what we call \emph{computable} map between groups. Namely, let $G= \langle X \rangle$, $H = \langle Y \rangle$ be two countable groups with enumerated sets of generators $X = \{ x_1, x_2, \ldots \}$ and $Y = \{ y_1, y_2, \ldots \}$, respectively. Then, a map $\phi: G \rightarrow H$ is called \emph{computable} (with respect to this setting), if there exists a Turing machine that for each input $i \in \mathbb{N}$ outputs an element from $(Y \cup Y^{-1})^*$ that represents the image of $x_i$ in $H$. Note that if $X$ is finite (respectively, $Y$ is finite), then computability of $\phi: G \rightarrow H$ implies the existence of a computable map between $G$ and $H$ with respect to any finite (enumerated) generating set of $G$ (respectively, of $H$). 
 In fact, it follows from the construction of our paper \cite{ourpaper} that  the embedding of Proposition \ref{prop-0} {can be made} computable. Also, note that for a finitely generated group with computable left-order, the induced left-order on  any computably embedded subgroup is also computable with respect to the given recursive enumeration of the embedded group. This immediately yields the following theorem.
 \begin{thm}
\label{thm-finitization}
A {countable group} has {a}  computable left-order if and only if it computably embeds into a finitely generated simple group with computable left-order.
\end{thm}

It would be interesting to know whether one can omit the condition of computability of the embedding in Theorem \ref{thm-finitization}. In case the embedded subgroup is finitely generated, this is true, because any embedding of a finitely generated group into a finitely generated group is computable. For the general case, we ask the following question.
\noindent
\begin{quest}
\label{quest-1}
 Does there exist a countable group with no computable left-order that embeds into a finitely-generated group with computable left-order?
\end{quest}
\noindent
\begin{rem}
     A left-orderable group without computable left-order but decidable word problem has been constructed in \cite{harrison_left-orderable_2018}. For a finitely generated example, see \cite{arman_new}. As follows from Theorem \ref{thm-finitization} and the observation below it, there is no finitely generated group that serves as a positive answer to Question \ref{quest-1}.
\end{rem}

 Another natural question about ``algorithmic knowledge'' is whether the embedding of $G$ into $H$ \emph{preserves the membership problem}. That is, given a finite word, formed by a given (symmetric) set of generators of {$H$}, whether or not that word represents an element of the image of $G$ in $H$? Informally speaking, decidability of the membership problem means that from the algorithmic point of view the group {$H$} ``knows'' where the image of the group {$G$} sits inside itself.  
 
 The decidability of the membership problem allows {one} to characterise computably left-orderable groups as follows (cf. Corollary 4 of \cite{arman_new}). 
 
  \begin{thm}
 \label{thm-new-1} \label{thm-new-2}
A countable group $G$ has a computable left-order if and only if there exists an embedding of $G$ into a finitely generated computably left-orderable simple group $H$ such that the membership problem for $G$ in $H$ is decidable.
 \end{thm}
\begin{rem}
In Theorem \ref{thm-new-1}, we can in addition require the embeddings to be Frattini, and isometric whenever $H$ is finitely generated. Theorem \ref{thm-new-1}  is a version of Theorem \ref{thm-finitization}, where  we do not need to require that the embedding  is computable. 
\end{rem}


\subsection{Stronger versions}

{Let $\mathcal{G} = (G_1, G_2, \ldots )$ be an enumerated family (i.e., collection) of finitely generated groups.} By a result of Higman-Neumann-Neumann \cite{higman_embedding_1949}, there is a finitely generated group $H$ such that  every group {in} $\mathcal{G}$ embeds into $H$. An embedding construction of Neumann, in addition, preserves left-orderability \cite{neumann_embedding_1960}. Theorem \ref{T: family computable left order} and \ref{T: family word problem} below can be regarded as strengthened extensions of these results for countable families of groups.

\begin{df}
{Suppose that the groups in $\mathcal{G}$ have recursive presentations. 
  For $i = 1, 2, \ldots$, let $\mathcal{TM}_i$ be a Turing machine that lists generators and corresponding relators of the group $G_i$, and let $\xi: \mathbb{N} \rightarrow \mathbb{N}$ be defined as $ \xi: i \mapsto \mathcal{TM}_i$. Then, if $\xi$ is a computable map, we call $\mathcal{G}$ a \emph{recursively enumerated family of groups}.  } 
 {If, in addition, the groups in $\mathcal{G}$ have decidable word problem and the Turing machines $\mathcal{TM}_i$ also solve the word problem of $G_i$, then we call the enumerated family of groups $\mathcal{G}$ \emph{recursively enumerated family of groups with decidable word problem.} } 
 {Analogously, we say that $\mathcal{G}$ is a \emph{recursively enumerated family of groups with computable left-order} if the groups in $\mathcal{G}$ have computable left-orders and for $i = 1, 2, 3, \ldots$, the Turing machine $\mathcal{TM}_i$, in addition, constructs a recursive enumeration of the positive-cone for some computable left-order of $G_i$.   }
\end{df}

\begin{thm}\label{T: family computable left order} 
For {every} countable family $\mathcal{G}$ of finitely generated left-orderable groups, there exists a left-orderable finitely generated simple group $H$ such that every $G\in \mathcal G$ quasi-isometrically and Frattini embeds into $H$. If, {in addition, $\mathcal{G}$ is a recursively enumerated family of groups with computable left-order}, then $H$ {has a computable left-order}.
\end{thm}
In \cite{ourpaper}, we also proved a finitely generated version of the Bludov-Glass theorem \cite{bludov_word_2009} for left-ordered groups with recursively enumerated positive cone{s}. Moreover, the flexibility of our method allowed us to prove a quasi-isometric version of a theorem of Thompson \cite{thompson_word_1980} about embeddings into finitely generated simple groups with decidable word problem. 
{The statement that every finitely generated group quasi-isometrically embeds in a finitely generated simple group also appears as Theorem C of \cite{belk_twisted_2022}.} 
Here we observe that a similar statement is true for recursively enumerable families as well.  
\begin{thm}\label{T: family word problem}
For every recursively enumerable family $\mathcal{G}$ of finitely generated groups with decidable word problem, there exists a finitely generated simple group $H$ with decidable word problem such that each $G\in \mathcal G$  both quasi-isometrically and Frattini embeds into $H$. 
\end{thm}
\begin{rem}
In {Theorem \ref{T: family word problem}}, the restriction that $\mathcal{G}$ is recursively enumerable cannot be replaced with the condition that $\mathcal{G}$ is just countable. This follows from a theorem of Boone and Rogers \cite{boone_problem_1966}, stating that there is no universal partial algorithm that solves the word problem in all finitely 
 {presented} groups with decidable word problem. {Moreover, in the statement of Theorem \ref{T: family word problem}, if we do not require $H$ to have decidable word problem, then $\mathcal{G}$ does not need to be recursively enumerable, just countable.}
\end{rem}
Combined with a theorem of Olshanskii \cite{olshanskii_distortion_1997}, {Theorem \ref{T: family word problem}}  implies a modified version of the famous Boone-Higman and Thompson's theorem on characterization of groups with decidable word problem. Namely, we have:
\begin{thm}\label{T: word problem quasi isometry}
A given finitely generated group has decidable word problem if and only if it quasi-isometrically embeds into a quasi-isometrically embedded finitely generated simple subgroup of a finitely presented group.
\end{thm}
\noindent
This answers a question asked by S. Kunnawalkam during the ``Beyond hyperbolicity'' conference.  {The analogous question for the left-orderable groups is left open.}

\begin{quest}
Is there a quasi-isometric version of Boone-Higman's theorem for left-orderable groups? 
\end{quest}

\noindent
\textbf{Acknowledgements.}
The authors have been supported in part by the Austrian Science Fund (FWF) project P35079-N. The authors also would like to thank the anonymous referee and Indira Chatterji for helpful remarks that improved the paper.

\section{Geometric embeddings of families of groups} \label{S:geomeric embeddings}

{We now discuss the proofs of Theorems \ref{T: family computable left order}  and \ref{T: family word problem}.} 
Let $\mathcal{G}$ be a countable family of finitely generated groups. 
The aforementioned construction of Neumann has been adapted to deal with various computability properties of embeddings into a finitely generated group \cite{darbinyan_group_2015}, see also \cite{arman_new}. 

\subsection{Geometric embeddings of countable families of groups}
\label{subsection 2.1.}

For a finitely generated group $G = \langle X \rangle$ and $g \in G$, by $|g|_X$ we denote the length of the shortest word in $(X \cup X^{-1})^*$ that represents $g$ in $G$. Let $G$ be generated by a finite set $X$ and $H$ be generated by a finite set $Y$. 
\begin{df}\label{D:quasi-isometry}
An embedding $\phi: G \hookrightarrow H$ is said to be a \emph{quasi-isometric embedding}  if there exist constants
$C, k > 0$ such that, for all $g\in G$, 
$$\frac{1}{C}|g|_X - k \leq |\phi(g)|_Y \leq C |g|_X + k.$$
\end{df}
Note that the constants $C$ and $k$ depend on the choice of the generating sets. 

\begin{prop}\label{P: qi arman} Let $\mathcal{G}$ be a countable family of finitely generated groups. Then there is a 2-generated group $\widetilde{H}$ such that  every group $G_i$ from $\mathcal{G}$ quasi-isometrically and Frattini embeds into $\widetilde{H}$. Moreover, if all groups in $\mathcal{G}$ are left-orderable, then so is $\widetilde{H}$. 
\end{prop}

To prove this proposition, we make some additional observations on the construction in \cite{darbinyan_group_2015}.  
 We agree on the following notation. {We use the following notation of \cite{darbinyan_group_2015}.} For group elements $x,y$, we write $x^y=yxy^{-1}$ and $[x,y]=xyx^{-1}y^{-1}$. 
 
Let $\mathcal{G}=\{G_1,G_2,\ldots \}$. For each group $G_i$ of $\mathcal{G}$ fix a finite generating set $X_i=\{x_{i_1}, \ldots, x_{i_{n_i}}\}$. Let 
$$
G=\bigoplus_i G_i.
$$ 
We fix $X=\bigcup_i X_i$ to be a generating set of $G$, and renumerate its elements in the most natural way so that $x_1=x_{1_1},x_2=x_{1_2}, \ldots$. 

We now define an embedding of $G$ into a 2-generated subgroup of the double \emph{unrestricted} wreath product $G\wr \mathbb{Z} \wr \mathbb{Z}$. This is done in two steps. 

First we embed $G$ into the unrestricted wreath product $G\wr \mathbb{Z}$. For this purpose, we let $\langle t\rangle\simeq \mathbb{Z}$ and define a function $f_i \in  G^{\langle t\rangle}$ to be the generator $x_i$ on all elements of strictly positive exponent in $\langle t\rangle$, and the identity elsewhere. 

Let $\psi:X\to G\wr\langle t\rangle$ be defined by $\psi(x_i)=[t,f_i]$. 
\begin{lem} The map $\psi$ {extends} to an embedding of $G$ into $G\wr \langle t\rangle$.
\end{lem}
\begin{proof}  Observe that $\psi(x)$ is equal to $x$ on $t^0$ and the identity elsewhere. Thus $\psi$ extends to an embedding of $G$. 
\end{proof}

Let us now use the embedding $\psi$ to embed $G$ into a 2-generated subgroup of $G\wr \langle t\rangle \wr \langle s\rangle$, where both $t$ and $s$ generate infinite cyclic groups. 

We first define a function $F\in (G\wr \langle t\rangle)^{\langle s\rangle}$ such that $F(s^k)$ is equal to $t$ for $k=1$, equal to $f_i$ for $k=2^i$, $ i>0$, and {to} the identity elsewhere. 

We then define $\Psi:X\to G\wr \langle t\rangle \wr \langle s\rangle$ by $\Psi(x_i)=[F,F^{s^{2^i-1}}]$. We let $\widetilde{H}$ be the subgroup of $G\wr \langle t\rangle\wr \langle s\rangle$ generated by $s$ and $F$. 

  \begin{lem} The map $\Psi$ extends to an embedding of $G$ into $\widetilde{H}$.
\end{lem}
\begin{proof} Observe that $\Psi(x_i)(s^k)$ is equal to $\psi(x_i)$ for $k=1$. Moreover, it is the identity elsewhere. Indeed, if   $\Psi(x_i)(s^k)\not =1$, then $F(s^k)\not =1$ and $F^{s^{2^i-1}}(s^k)\not=1$. By definition of $F$ this implies that $k$ and $k+2^i-1$ are both {powers} of $2$. This is not possible unless $k=1$. 
\end{proof}

\begin{lem} The embedding $\Psi:G\to \widetilde{H}$ restricts to a Lipschitz embedding $\Psi_l:G_l\to \widetilde{H}$, for $l \in \mathbb{N}$. More precisely, let  
$M_l$ be the largest index of the generators of $G_l$. Then, for all $g\in G_l$, 
$$ |g|_{X_l} \leqslant |\Psi (g)|_{\{s,F\}} \leqslant 2^{M_l+2}|g|_{X_l}.$$  
\end{lem}

\begin{proof}
The right hand side of the inequality is by definition of the embedding $\Psi$. To prove the left hand side we proceed as follows. 
Let 
$$w=s^{\alpha_1}F^{\beta_1}s^{\alpha_2}F^{\beta_2}\cdots s^{\alpha_{n-1}}F^{\beta_{n-1}}s^{\alpha_n}$$ be a word of minimal length  that represents $\Psi(g)$, where $g\in G_l$, $l\in \mathbb{N}$. 

 As  $s^{\alpha}F=(F^{s^{\alpha}})s^{\alpha}$, the word $w$ is freely equal to a word of the form
$$
w'=(F^{s^{\gamma_1}})^{\epsilon_1}(F^{s^{\gamma_2}})^{\epsilon_1}\cdots (F^{s^{\gamma_{N}}})^{\epsilon_N}s^\gamma,
$$
where $\gamma_1, \ldots, \gamma_N \in \mathbb{Z}$, $\epsilon_1, \ldots, \epsilon_N \in \{\pm 1\}$.
Without loss of generality, we assume that $(\gamma_i, \epsilon_i) \neq (\gamma_{i+1}, -\epsilon_{i+1})$, $i = 1, \ldots, N-1$, so that $|w|, |w'| \geq N$.

 As $w$ represents $\Psi(g)$ in $\widetilde{H}$, it must be that $\gamma = 0$.   Let $y_i = (F^{s^{\gamma_i}})^{\epsilon_i}(s^1) \in  G \wr \langle t \rangle$. Then, depending on the value of $\gamma_i$, $y_i$ is either equal to $t^{\pm 1}$ or to $(f_j)^{\pm 1} \in G\wr \langle t\rangle$ for some $j \in \mathbb{N}$ or it is equal to the identity.
 Therefore, by using the relation $tf_i =f_i^t t$, one obtains that 
 $$y_1\cdots y_N = (f_{n_1}^{t^{m_1}})^{\epsilon_1}(f_{n_2}^{t^{m_2}})^{\epsilon_2} \ldots (f_{n_s}^{t^{m_s}})^{\epsilon_s}t^m,$$
for some $0 \leq s \leq N$, $n_1, \ldots, n_s \in \mathbb{N}$, $m, m_1, \ldots, m_s \in \mathbb{Z}$, and $\epsilon_1, \ldots, \epsilon_s \in \{\pm 1\}$.

  Recall that $\Psi(x_i)(s^1)$ is equal to $\psi(x_i)$. Thus $\Psi(g)(s^1)=y_1\cdots y_N$ is equal to $\psi(g)$. By definition of $\psi$, we conclude that $m=0$.

Finally, since $f_{n_1}^{t^{m_1}} (t^0)$ is either $1_{G_l}$ or a generating element from $X_l$, we get that $y_1 \ldots y_N (t^0) = g$ implies $s \geq |g|_{X_l}$. Therefore, $N \geq  |g|_{X_l}$, hence $ |g|_{X_l} \leqslant |\Psi (g)|_{\{s,F\}}$.
 \end{proof}

 The embedding of $G$ into $\widetilde{H}$ is Frattini \cite[Theorem 5.15(4)]{ourpaper}. As the embedding of $G_i$ into $G$ is Frattini, so is the embedding of $G_i$ into $\widetilde{H}$.  This finishes the proof of Proposition \ref{P: qi arman}.

\subsection{Computability aspects} We now assume that $\mathcal{G}$ is a recursively enumerated family of groups with decidable word problem.
We now comment on computability properties of the embedding of Proposition \ref{P: qi arman} under this additional assumption. 

{A} countable (but not necessarily finitely generated) group is \emph{computable} with respect to a fixed recursively enumerated generating set if there exists an algorithm that decides whether or not a finite word, formed by the generators and their inverses, represents the trivial element of the group. Observe that $$
G=\bigoplus_i G_i
$$ 
is a computable group with respect to the generating set $X=\bigoplus_i X_i$, because $\mathcal G$ is a recursively enumerated family of groups with decidable word problem.  
If, in addition, $\mathcal G$ is a recursively enumerated family of groups with computable left-order, then $G$ is computably left-orderable.

\begin{prop}\label{P: geometric and computable} Let $\mathcal{G}$ be a countable family of finitely generated groups. Let $G$ be defined as above, and let $\Psi: G \hookrightarrow \widetilde{H}$ be defined as in Subsection \ref{subsection 2.1.}. Then the following statements holds: 
\begin{enumerate}
    \item If $\mathcal G$ is recursively enumerated, then the induced embedding {of $G_i$ into $\widetilde{H}$} is computable for every $G_i \in \mathcal G$.
    \item If $\mathcal{G}$ is a recursively enumerated family of finitely generated groups with decidable word problem, then $\widetilde{H}$ has decidable word problem, and the membership problem for the image of $G_i$ in $\widetilde{H}$ is {decidable} for every group $G_i$ in $\mathcal{G}$.
    \item  If $\mathcal{G}$ is a recursively enumerated family of finitely generated groups with a computable left-order, then $\widetilde{H}$ has a computable left-order.
\end{enumerate}
\end{prop}
\begin{proof} We apply the results of \cite{darbinyan_group_2015,arman_new,ourpaper} to the embedding $\Psi:G\to \widetilde{H}$: the  embedding  is computable, 
 \cite[Theorem 3(a)]{arman_new}, hence, the induced embedding of $G_i$ into $\widetilde{H}$ is computable for all $G_i\in \mathcal{G}$. Under the assumption of assertion (2) of the proposition, $G$ is computable with respect to the generating set $X$, hence, $\widetilde{H}$ has decidable word problem  \cite{darbinyan_group_2015}. 
 Moreover, the membership problem for the image of $G$ in $\widetilde{H}$ is decidable \cite[Theorem 3(e)]{arman_new}. As the membership problem for the groups $G_i$ in $G$ is decidable, then the membership problem for the image of $G_i$ in $\widetilde{H}$ is decidable. 
  Alternatively, as the groups $G_i$ quasi-isometrically embed in $\widetilde{H}$ by Proposition \ref{P: qi arman} the membership problem is decidable, see Lemma \ref{L: qi membership} below.
Under the assumptions of (3) of the proposition, $G$ is computably left-ordered, hence, $\widetilde{H}$ is computably left-ordered  \cite[{Theorem 3(f)}]{arman_new}.
\end{proof}

\begin{proof}[Proof of Theorem \ref{T: family computable left order}]
{Let $\mathcal{G}$ be an enumerated family of finitely generated left-orderable groups. Let $\widetilde{H}$ be the finitely generated left-ordered group given by Proposition \ref{P: geometric and computable}. By \cite[Theorem 1]{ourpaper} there is a quasi-isometric Frattini embedding of $\widetilde{H}$ into a finitely generated simple left-orderable group $H$. Since by Proposition 2.1 the embedding of every group $G_i$ from $\mathcal{G}$ into $\widetilde{H}$ is quasi-isometric and Frattini, so is the embedding of $G_i$ into the finitely generated simple group  $H$. If $\mathcal{G}$ is a recursively enumerated family of finitely generated computably left-ordered groups, then $G$ has a computable left-order. Theorem 2 of \cite{ourpaper} then implies that $H$ can be chosen such that, in addition, $H$ has a computable left-order.
}
\end{proof}

\begin{proof}[Proof of Theorem \ref{T: family word problem}]
{
Let $\mathcal{G}$ be a recursively enumerated family of groups with decidable word problem. Let $\widetilde{H}$ be the finitely generated group with decidable word problem given by Proposition \ref{P: geometric and computable}. By \cite[Theorem 3]{ourpaper} there is a quasi-isometric Frattini embedding of $\widetilde{H}$ into a finitely generated simple group $H$ with decidable word problem. As the embedding of every group $G_i$ in $\mathcal{G}$ into $\widetilde{H}$ is quasi-isometric and Frattini, so is the embedding of $G$ into $H$.
}
\end{proof}

\section{Proof of Theorem \ref{thm-new-1}}

 In this section we discuss the proof of Theorem \ref{thm-new-1}.
 
\begin{lem}\label{L: qi membership}
Let $G$ be a finitely generated group. If $G$ quasi-isometrically embeds into a finitely generated group $H$ with decidable word problem, then the membership problem for $G$ in $H$ is decidable.  
\end{lem}

\begin{proof}  Suppose that the embedding is $(C,k)$-quasi-isometric. Let $Y$ be a finite symmetric generating set for $H$ and let $w$ be a word in $Y$, representing an element $\overline{w}$ of $H$.  Compute the length $|\overline{w}|_Y$. This is possible as the word problem in $H$ is decidable. Make a list of all words in the generators $G$ of length in $[C^{-1}|\overline{w}|_Y-k,C|\overline{w}|_Y+k]$ and check whether or not $w$ is equal to the embedded image of a  word in this list. If it is, $w$ represents a word in $G$. Otherwise, it does not.
\end{proof}

\begin{proof}[Proof of Theorem \ref{thm-new-1}]  Let {$G$} 
  be a computably left-ordered  group. {Then, by Corollary 4 of \cite{arman_new}, $G$ can be embedded into a finitely generated computably left-ordered group $G_1$  such that the membership problem for $G$ in $G_1$ is decidable.} By Theorem \ref{thm-0}, $G_1$ quasi-isometrically embeds into a finitely generated simple group $H$ that is computably left-ordered. Recall that this implies that the word problem in $H$ is decidable. Additionally, {by Lemma \ref{L: qi membership}, the membership problem for $G_1$ in $H$ is decidable, which implies that the membership problem for $G$ in $H$ is decidable, too.}

 Inversely, suppose that $H$ has a computable left-order and that the membership problem for $G$ in $H$ is decidable. In this case, there is the following algorithm to show that the order on $G$ induced by $H$ is computable.   First, {list} all elements of $G$ in the following way: the first element of our list is the identity element. We {list} all elements of $H$ and represent each element by a unique word $w$. This is possible by decidability of the word problem. For each element $w$ of this list we proceed as follows. As the membership problem for $G$ in $H$ is decidable, we check whether $w$ represents an element of $G$. If it does, we add it to the list of elements of $G$.  Otherwise, we continue with the next word.   
  
Let us denote the list obtained this way by $L_G$. 
  For every word in  $L_G$, we can now decide whether this word represents the identity, a positive or a negative element. This means that the left-order on $G$ induced from the left-order of $H$ is computable with respect to $L_G$.
\end{proof}

\mbox{}

 \addtocontents{toc}{\setcounter{tocdepth}{-10}}
 
\bibliographystyle{alpha}

\bibliography{eilos}

\end{document}